\documentclass[12pt]{article}
\usepackage{amsmath, amssymb, amsthm, enumerate}
\usepackage{mathtools}

\newtheorem{theorem}{Theorem}[section]
\newtheorem{lemma}[theorem]{Lemma}

\newtheorem{remarks}[theorem]{Remarks}
\newtheorem{cor}[theorem]{Corollary}
\newtheorem{defin}[theorem]{Definition}

\def\rr{{\mathbb R}}

\def\zz{{\mathbb Z}}
\def\qq{{\mathbb Q}}

\def\kk{{\cal K}}
\def\sik{{\rr}^2} 
\def\am{^{-1}} 
\def\su{\subset}
\def\sp{\supset}
\def\se{\setminus}
\def\al{\alpha}
\def\be{\beta}
\def\ga{\gamma}
\def\de{\delta}
\def\ep{\varepsilon}
\def\la{\lambda}
\def\si{\sigma}

\def\cd{\cdot}
\def\stb{,\ldots ,}
\def\emp{\emptyset}
\def\sumin{\sum_{i=1}^n}

\def\cl{{\rm cl}\, }
\def\nl{[0,1]}
\def\proof{\noindent {\bf Proof.} }
\def\graph{{\rm graph}\, }
\def\dist{{\rm dist}\, }
\def\pr{{\rm pr}_x \, }
\def\akkor{\Longrightarrow}

\begin{document}
\title{Graph-null sets}
\author{M. Laczkovich and A. M\'ath\'e}

\footnotetext[2]{{\bf MR subject classification:} 28A75}

\maketitle

\begin{abstract}
We say that a plane set $A$ is {\it graph-null,} if there is a function 
$g\colon \nl \to \rr$ such that $\la _2 (A+\graph g)=0$. A plane set $A$ has
the {\it translational Kakeya property} if, for every translated copy $A'$
of $A$ and for every $\ep >0$, there is a finite sequence of vertical
and horizontal translations bringing $A$ to $A'$ such that the area touched
during the horizontal translations is less than $\ep$.

These properties are equivalent if $A$ is compact. We show that 
the graph of every absolutely continuous function is graph-null. Also, the graph
of a typical continuous function is graph-null. Therefore, there are nowhere
differentiable continuous functions whose graphs are graph-null. Still,
we show that there exists a continuous function whose graph is not graph-null.
\end{abstract}

\section{Introduction and main results}\label{s1}

We say that a set $A\su \sik$ has the
{\it Kakeya property} (shortly property (K)), if there exists a set $A'\ne A$
such that $A$ can be continuously moved to $A'$ within a set of
arbitrarily small area. Clearly, every segment and every circular arc has
property (K). A set $A\su \sik$ is said to have the {\it strong Kakeya
property} (shortly property ${\rm (K^s )}$) if, whenever $A'$ is an arbitrary
congruent copy of $A$, then $A$ can be continuously moved to $A'$ within a set
of arbitrarily small area. It is well known that line segments have 
property ${\rm (K^s )}$; this follows from Besicovitch' solution of
the classical Kakeya problem \cite{B}, \cite[Theorem 7.6]{F}. Some circular
arcs also have  property ${\rm (K^s )}$. For short enough arcs this was proved
in \cite{HK}, and for circular arcs shorter than the half circle in \cite{CC}.
It is not known whether or not every proper circular arc has property
${\rm (K^s )}$. On the other hand, if a continuum has property (K), then it
must be a segment or a circular arc \cite{CHL}.
In this note we investigate a related notion.
\begin{defin} {\rm A plane set $A$ has the} translational Kakeya property
${\rm (K^t)}$ {\rm if, for every translated copy $A'$ of $A$ and for every
$\ep >0$, there is a finite sequence of vertical and horizontal translations
bringing $A$ to $A'$ such that the area touched during the} horizontal {\rm
translations is less than $\ep$.}
\end{defin}

The most interesting cases are the continuous curves. 
We will see that property ${\rm (K^t)}$ is shared by the graphs of most
continuous functions and by most absolutely continuous curves.
As paradoxical as it may seem, every circle and every polygon without
vertical sides has property ${\rm (K^t)}$ (see Remark \ref{r1}.1).

First we give an equivalent formulation of property ${\rm (K^t)}$.
We denote the $d$ dimensional Lebesgue outer measure by $\la _d$.
When a set $A\su \sik$ is translated by the vector $(a,0)$, the set of points
touched is $A+ ([0,a]\times \{0\} )$. From this observation it is clear
that {\it a set $A\su \sik$ has property ${\rm (K^t)}$ if and only if, for
every $\ep >0$  there is a simple function $h\colon \nl \to \rr$ such that 
$\la _2 (A+\graph h)<\ep$.}
(A function $h\colon [a,b]\to \rr$ is {\it simple}, if there is a partition
$a=x_0 <x_1 <\ldots <x_k =b$ such that $h$ is constant in $(x_{i-1} ,x_i )$
for every $i=1\stb k$.)

If the projection of a set $A$ onto the $y$ axis is of measure zero, then
$A$ has property ${\rm (K^t)}$. Indeed, translating $A$
horizontally, the area touched has two dimensional measure zero. 

In the other direction, let $E_A$ denote the set of numbers $y\in \rr$ such
that $((0, \ep )\times \{ y\})\cap A \ne \emp$ for every $\ep >0$. If
$A$ has property ${\rm (K^t)}$, then $E_A$ must be of measure zero. Indeed,
if $h\colon \nl\to \rr$ is a simple function, and $0=x_0 <x_1 <\ldots <x_k =1$
is a partition such that $h(x)=c_i$ if $(x_{i-1} <x<x_i )$ for every $i=1\stb k$,
then $A+\graph h$ contains the set $(x_{i-1} ,x_i ) \times (E_A +c_i )$ for every
$i=1\stb k$, and thus
$$\la _2 (A+\graph h)\ge \sumin \la _1 (E_A )\cd (x_i -x_{i-1} )=\la _1 (E_A ).$$

We also consider the following variant:
\begin{defin}
{\rm A plane set $A$ is} graph-null, {\rm if there is a function 
$g\colon \nl \to \rr$ such that $\la _2 (A+\graph g)=0$.}
\end{defin}
If the projection of the set $A$ onto a nonhorizontal line $e$ is of linear
measure zero, then $A$ is graph-null. Indeed, in this case $\la _2 (A+f)=0$
for every line $f$ perpendicular to $e$. 

For example, the set $B=\{ (x,y)\colon y-x\in \qq \}$ is graph-null. On the
other hand, $B$ does not have property ${\rm (K^t)}$, since $((0,\ep )\times
\{ y\} )\cap B\ne \emp$ for every $\ep >0$ and $y\in \rr$. Therefore, the
graph-null property does not imply ${\rm (K^t)}$. In Remark \ref{r1}.2 we
will see that the graph-null property does not imply ${\rm (K^t)}$ even
for the graphs of functions that are continuous everywhere except one point.
One can also show that the ${\rm (K^t)}$ property does not imply the
graph-null property even for $G_\de$ sets (see \cite{M}). Therefore,
${\rm (K^t)}$ and the graph-null property are independent.

We show, however, that for compact sets property ${\rm (K^t)}$
is equivalent to the graph-null property, moreover, it is equivalent to an
even more restrictive property defined as follows.

Let $K(\nl ^2 )$ denote the family of nonempty compact subsets of $\nl ^2$.
It is well known that the Hausdorff metric
$$d(K_1 ,K_2 )=\max \{ \dist (x,K_1) , \dist (y,K_2) \colon x\in K_2 , \
y\in K_1 \}$$
$(K_1 ,K_2 \in K(\nl ^2 ))$  makes $K(\nl ^2 )$ a compact metric space
\cite[Theorem 4.26, p. 26]{K}. Let $\kk =\{ K\in K(\nl ^2 ) \colon
{\rm pr}_x K =\nl \}$, where  ${\rm pr}_x K$ denotes the projection of $K$ onto
the $x$ axis. It is easy to check that $\kk$ is a closed, hence compact subset
of $K(\nl ^2 )$.
\begin{defin}
{\rm We say that a set $A\su \sik$ is {\it typically graph-null}, if
$\la _2 (A+K)=0$ holds for every $K$ belonging to a comeager subset of $\kk$.}
\end{defin}
Since $\pr K=\nl$ for every $K\in \kk$, it follows that every element of $\kk$
contains the graph of a function defined on $\nl$. Therefore, every typically
graph-null set is graph-null.
\begin{theorem} \label{t1}
For every compact set $A\su \sik$, the following are equivalent.
\begin{enumerate}[{\rm (i)}]
\item $A$ has property ${\rm K}^t$.
\item For every $\ep >0$  there is a function $g\colon \nl \to \rr$ such that 
$\la _2 (A+\graph g)<\ep$.
\item $A$ is graph-null.
\item $A$ is typically graph-null.
\end{enumerate}  
\end{theorem}
We will prove Theorem \ref{t1} in Section \ref{s2}.
As we mentioned already, the graphs of most continuous functions share
the properties listed in Theorem \ref{t1}. First we show that smooth functions
have this property.
\begin{theorem} \label{t2}
If $f\colon [a,b] \to \rr$ is continuously differenti\-able, then
\hbox{$\graph f$} is graph-null.  
\end{theorem}
\proof
The statement is a consequence of the following theorem of E. Sawyer
\cite{S}: {\it there exists a function $\phi \colon \rr \to \rr$ such that
whenever $f\colon [a,b] \to \rr$ is continuously differentiable, then the
range of $\phi |_{[a,b]} -f$ is of measure zero.}

The function $\phi$ constructed by Sawyer is everywhere continuous except at
the points of a countable set $C$. It follows that $\graph \phi$ is a Borel
subset of $\sik$.

Let $f\colon [a,b] \to \rr$ be continuously differentiable. Since 
$\graph f$ is compact, $(\graph f) \times (\graph \phi )$ 
is a Borel subset of the product space $\sik \times \sik$. Then the set
$E=(\graph f)+ ( \graph \phi )$, being a continuous image of $(\graph f) \times
(\graph \phi )$ is analytic, hence Lebesgue measurable. We have 
$$E= \graph f +\graph \phi =\{ (x_1 +x_2 , f(x_1 )+\phi (x_2 )) \colon
x_1 \in [a,b],\ x_2 \in \rr \} .$$
Then, for every $c\in \rr$ we have
$$E_c =\{ y\colon (c,y)\in E\} =\{ f(c-x)+\phi (x)\colon c-x\in [a,b]\} ,$$
which is the range of the function $\phi |_{[c-b,c-a]} -g$, where $g(x)=-f(c-x)$
$(x\in [c-b,c-a])$. Since $g$ is continuously differentiable, it follows
from Sawyer's theorem that $\la _1 (E_c )=0$ for every $c$. Then, by Fubini's
theorem, we have $\la _2 (E)=0$, proving that $\graph f$ is
graph-null. \hfill $\square$

\begin{cor}\label{c1}
If a set $A\su \sik$ can be covered by countably many graphs of 
continuously differentiable functions, then $A$ is graph-null.
\end{cor}
Indeed, we have $\la _2 (A+\graph \phi )=0$, where $\phi$ is Sawyer's function.

\begin{remarks}\label{r1}
{\rm 1.
It is clear that every circle can be covered by countable many graphs of 
continuously differentiable functions. Therefore, {\it every circle is
graph-null and, consequently, has property ${\rm K^t}$} (by Theorem \ref{t1}).
The same is true for polygons without vertical sides.

2. Let $f(x)=\sin (1/x)$ if $x\in (0,1]$, and let $f(0)=0$. Then $\graph f$
is graph-null, as it can be covered by countable many graphs of continuously
differentiable functions. On the other hand, $\graph f$ does not have
property ${\rm K}^t$, since $\graph f$ intersects $(0,\ep )\times \{ y\}$
for every $\ep >0$ and $y\in [-1,1]$. This example shows that the graph-null
property does not imply ${\rm K}^t$ even for graphs of functions that are
continuous everywhere except one point.

3. It is clear that if a set $A$ has any of the properties (i)-(iv) listed
in Theorem \ref{t1}, then every subset of $A$ has the same property. Thus we
may ask whether or not the family of sets having any of these properties form
an ideal, or even a $\si$-ideal.

The answer is positive in the case of property (iv). The family of
typically graph-null sets is a $\si$-ideal, since the intersection of countable
many comeager subsets of $\kk$ is also comeager in $\kk$.

The family of sets having property ${\rm K}^t$ is not a $\si$-ideal. Indeed,
let $f$ be as in Remark 2 above, and let $f_n$ be the restriction of $f$
to the interval $[1/n,1]$. Then $\graph f_n$ has property ${\rm K}^t$ by
Theorem \ref{t2} for every $n=2,3,\ldots$. On the other hand,
$\bigcup_{n=2}^\infty \graph f_n =(\graph f) \se \{ (0,0)\}$ does not have
property ${\rm K}^t$.

It is not clear if the family of sets having property ${\rm K}^t$ is an ideal
or not. It is also open whether the sets having properties (ii) and (iii) of
Theorem \ref{t1} form ideals. 
}
\end{remarks}  
    
The following theorem extends Theorem \ref{t2} to absolutely continuous curves.
\begin{theorem}\label{t3}
Let $\ga =(\ga _1 ,\ga _2 )\colon [a,b] \to \sik$ be a curve such that
$\ga _1$ and $\ga _2$ are absolutely continuous functions on $[a,b]$. If 
$\ga '_1 (t)\ne 0$ for a.e. $t\in [a,b]$, then  $\ga ([a,b])$ is graph-null.
\end{theorem}
\begin{cor}\label{c2}
If $f\colon [a,b]\to \rr$ is absolutely continuous, then its graph is
graph-null. 
\end{cor}
The proof of Theorem \ref{t3}, provided in Section \ref{s3}, does not use
Sawyer's theorem.

The statement that
the graph of ``most continuous functions'' is graph-null is made precise
by the following theorem.
\begin{theorem}\label{t4}
The set of continuous functions $f\in C[a,b]$ such that $\graph f$ is
graph-null is comeager in $C[a,b]$.
\end{theorem}
\proof
It is easy to see that the set $U_\ep =\{ (f,K)\in C[a,b] \times \kk \colon 
\la _2 ((\graph f)+K)<\ep \}$ is open in $C[a,b] \times \kk$ for every
$\ep >0$. Since the set of continuously differentiable functions is dense in
$C[a,b]$, and every such function is typically graph-null by Theorems \ref{t1}
and \ref{t2}, it follows that $U_\ep$ is everywhere dense in $C[a,b]
\times \kk$ for every $\ep >0$. Therefore, the set
$$U =\{ (f,K)\in C[a,b] \times \kk \colon 
\la _2 ((\graph f)+K)=0 \} =\bigcap_{n=1}^\infty U_{1/n}$$
is a dense $G_\de$ set in $C[a,b] \times \kk$. Then, by the Kuratowski-Ulam
theorem \cite[Theorem 8.41, p. 53]{K}, there is a comeager subset $V\su C[a,b]$
such that, for every $f\in V$, $\{ K\in \kk \colon \la _2 ((\graph f)+K)=0 \}$
is comeager in $\kk$. \hfill $\square$

Since a typical continuous function is nowhere differentiable, we obtain that
{\it there are nowhere differentiable continuous functions such that their
graph is graph-null.} We show, however, that the graph-null property, although
is typical among continuous functions, is not shared by all continuous
functions.
\begin{theorem}\label{t5}
There exists a continuous function $f\in C\nl$ such that $\graph f$
is not graph-null.
\end{theorem}
Theorem \ref{t5} will be proved in Section \ref{s4}.

Let $M[a,b]$ denote the set of functions that are continuous and increasing
on $[a,b]$. Then $M[a,b]$ is a closed subspace of $C[a,b]$ and, consequently,
it is a Polish space under the supremum metric.
\begin{theorem}\label{t6}
The set of increasing continuous functions $f\in C[a,b]$ such that $\graph f$
is graph-null is comeager in $C[a,b]$.
\end{theorem}

The proof is the same as that of Theorem \ref{t4}. It is known that
a typical element of $M[a,b]$ is singular; that is, its derivative is
zero a.e. Therefore, we find that {\it there are singular, increasing and
continuous functions such that their graph is graph-null.}

We conclude this section with the following open question:

{\it Is the graph of every monotonic and continuous function graph-null?}

\section{Proof of Theorem \ref{t1}}\label{s2}
We denote by $\cl A$ the closure of the set $A$. 
\begin{lemma} \label{l1}
If a bounded set $A$ has property ${\rm K}^t$, then for every $\ep >0$ there
exists a simple function $h\colon \nl \to [0,\ep ]$ such that
$\la _2 (A+\graph h) <\ep$.
\end{lemma}

\begin{proof}
For $x\in \rr$ and $\ep>0$ we define
$$x \bmod \ep = x-\ep\lfloor x/\ep \rfloor,$$
so that $(x \bmod \ep) \in [0,\ep)$. Let $\Phi:\rr^2 \to \rr^2$ be the function
given by
$$\Phi(x,y)=(x, \,y \bmod \ep).$$
We show that $\la _2 (\Phi (E))\le \la _2 (E)$ for every $E\su \sik$. Put
$$S_i =\{ (x,y)\colon x\in \rr , \ y\in [i\ep ,(i+1)\ep )\}$$
for every $i\in \zz$. Clearly, $\Phi (E)=\bigcup_{i\in \zz} ((E\cap S_i )-
(0,i\ep ))$. Since the sets $S_i$ are measurable and disjoint, we have
$\la _2 (E)= \sum_{i\in \zz} \la _2 (E\cap S_i )$, and thus $\la _2 (\Phi (E))
\le \sum_{i\in \zz} \la _2 (E\cap S_i ) = \la _2 (E)$.

As $A$ is bounded, there is a positive integer $N$ such that
$$A\subset \rr\times [-N\ep, \,N\ep).$$
Since $A$ has property ${\rm K}^t$, there exists a simple function
$g\colon \nl \to \rr$ such that $\la _2 (A+\graph g) <\ep/(2N+1)$.
Let $h\colon [0,1]\to [0,\ep)$ be the (simple) function given by $h(x)=g(x) \bmod \ep$.

Notice that 
$$A+\graph h \subset \rr\times [-N\ep, \,(N+1)\ep),$$
hence $2N+1$ many translates of $\Phi(A+\graph h)=\Phi(A+\graph g)$ cover $A+\graph h$.
By translation invariance and subadditivity of Lebesgue outer measure, we obtain
\begin{align*}
\la_2(A+\graph h) &\le (2N+1) \la_2(\Phi(A+\graph g))\\
& \le (2N+1) \la_2(A+\graph g) \\
& < (2N+1) \frac{\ep}{2N+1} =\ep. 
\end{align*}
\qed
\end{proof}

\begin{lemma} \label{l2}
If a bounded set $A$ has property ${\rm K}^t$, then for every simple function
$g\colon \nl \to \rr$ and for every $\ep >0$ there exists a simple function
$h\colon \nl \to \rr$ such that $|g-h|<\ep$, and
$\la _2 (A+\graph h) <\ep$.
\end{lemma}
\proof
Let $0=x_0 <x_1 <\ldots <x_n =1$ be a partition of $\nl$ such that $g$
is constant in the interval $(x_{i-1},x_i )$ for every $i=1\stb n$.
By Lemma \ref{l1}, there is a simple function $f\colon \nl \to [0,\ep ]$ such
that $\la _2 (A+\graph f) <\ep /n$. Let $h=g+f$. Then $h$ is a simple function,
and $|h-g|<\ep$.

The set $A+\graph h$ is covered by the sets $A+(\graph h|_{(x_{i-1},x_i )})$
$(i=1\stb n)$, and by a finite number of translated copies of $A$. Therefore,
we have
$$\la _2 (A+\graph h) \le \sum_{i=1}^n \la _2 \left( A+(\graph h|_{(x_{i-1},x_i )})
\right) .$$
Since $A+(\graph h|_{(x_{i-1},x_i )})$ is a translated copy of
$A+(\graph f|_{(x_{i-1},x_i )}) \su A+\graph f$, we have 
$\la _1 (A+(\graph h|_{(x_{i-1},x_i )} ))<\ep /n$ for every $i=1\stb n$. Therefore,
$\la _2 ( A+\graph h)<\ep$. \hfill $\square$

We say that $K\in \kk$ is {\it $m$-simple} ($m\in \rr$), if there is a partition
$0=x_0 <x_1 <\ldots < x_n =1$ and there are closed segments $s_1 \stb s_n$ of
slope $m$ such that $\pr s_i =[x_{i-1},x_i ]$ for every $i=1\stb n$, and
$K=\bigcup_{i=1}^n s_i$.
\begin{lemma} \label{l3}
For every $m\in \rr$, the family of $m$-simple sets is everywhere dense in
$\kk$.
\end{lemma}
\proof
Let $m\in \rr$, $K\in \kk$ and $\ep >0$ be given. We show that there
exists an $m$-simple set $L$ such that $d(K,L)<3\ep$. Put $\eta =\ep /\max (1,
|m|)$.

Using the fact that $K$ is totally bounded, it is easy to see that there is a
finite set $F\su \nl ^2$ such that $d(K_0 ,F)<\ep$. Moving some points of $F$
slightly if necessary, we may assume that the $x$-coordinates of the
elements of $F$ are distinct. Let $F=\{ (x_0 ,y_1 ) \stb (x_n ,y_n ) \}$, where
$0\le x_0 <x_1 <\ldots <x_n \le 1$. Adding some points of $K$ to $F$ we may also
assume that $x_0 =0$, $x_n =1$, and $x_i -x_{i-1} < \eta$ for every $i=1\stb n$.
For every $i=1\stb n$ let $s_i$ be a segment such that its slope is $m$,
$\pr s_i =[x_{i-1} ,x_i ]$, $s_i \su \nl ^2$, and its left endpoint is $(x_{i-1} ,
z_{i-1})$, where $|z_{i-1} -y_{i-1}|<\ep$. Since $x_i -x_{i-1} < \eta$ and
$\eta \cd |m|\le \ep$, it is clear that there is such a segment for every
$i=1\stb n$.

Putting $L=\bigcup_{i=1}^{n} s_i$, we have $d(L,F)<2\ep$ and
$d(K,L)<3\ep$. \hfill $\square$

{\bf Proof of Theorem \ref{t1}.}
(i)$\akkor$(iv): It is enough to show that for every set $K_0 \in \kk$ and
for every $\ep >0$ there is a $K\in \kk$ such that $d(K,K_0 )<\ep$ and
$\la _2 (A+K) <\ep$. Indeed, if this is true, then the set $\{ K\in \kk \colon
\la _2 (A+K)<1/n\}$ is a dense open subset of $\kk$ for every $n$,
and then $\{ K\in \kk \colon \la _2 (A+K)=0\}$ is a dense $G_\de$ in $\kk$.

By Lemma \ref{l3}, there is an $m$-simple set $L$ with $m=0$ such that
$d(L,K_0 )<\ep /2$. Clearly, there is a simple function $g$ such that
$L=(\graph g)\cup F$, where $F$ is finite. By Lemma \ref{l2}, there is a simple
function $h\colon \nl \to \rr$ such that $|g-h|<\ep /2$, and
$\la _2 (A+\graph h) <\ep$. Put $K= \cl (\graph h) \cup F$. Then $K\in \kk$,
$d(K,K_0 )<\ep$, and $\la _2 (A+K)=\la _2 (A+\graph h) <
\ep$. 

The implications (iv)$\akkor$(iii) and (iii)$\akkor$(ii) are obvious.

(ii)$\akkor$(i): Let $\ep >0$ be given. If $g\colon \nl \to \rr$ is such
that $\la _2 (A+\graph g)<\ep$, then there is an open set $G$ such that
$A+\graph g\su G$ and $\la _2 (G)<\ep$. Then, for every $x\in \nl$ we have
$A+(x,g(x))\su G$. Since $A$ is compact and $G$ is open, it follows that
there exists a positive number $\de _x$ such that
$A+(t,g(x))\su G$ for every $t\in (x-\de _x ,x+\de _x )$.

By Cousin's lemma \cite[Theorem 1.3, p. 9]{BBT}, there exists a partition
$0=t_0 <t_1 <\ldots <t_n =1$, and for every $i=1\stb n$ there is a point
$x_i \in [t_{i-1},t_i ]$ such that $(t_{i-1},t_i )\su (x_i -\de _{x_i} ,
x_i +\de _{x_i} )$. Then we have $A+(t,g(x_i ))\su G$ for every $t\in
(t_{i-1},t_i )$ $(i=1\stb n)$.

Thus $A+\graph h\su G$, where $h(x)=g(x_i )$ $(x\in (t_{i-1},t_i )$,
$i=1\stb n)$. \hfill $\square$

\section{Proof of Theorem \ref{t3}}\label{s3}
We denote by $\la _1$ the linear measure ($1$ dimensional Hausdorff measure)
on $\sik$.
\begin{lemma} \label{l4}
Let $\ga \colon [a,b]\to \sik$ be a curve with absolutely continuous coordinate
functions. Suppose that for every $\ep >0$ there are pairwise disjoint
open intervals $J_1 ,J_2 ,\ldots \su [a,b]$ and real numbers $m_i ,c_i$
such that
\begin{enumerate}[{\rm (i)}]
\item $\sum_{i=1}^\infty |J_i |=b-a$,
\item for every segment $I\su \sik$ of slope $m_i$ we have
$\la _2 (\ga (\cl J_i )+I)\le c_i \cd |I|$ $(i=1,2,\ldots )$, and
\item $\sum_{i=1}^\infty c_i \cd \sqrt{1+m_i^2} <\ep$.
\end{enumerate}
Then $\ga ([a,b] )$ is graph-null. 
\end{lemma}
\proof
By Theorem \ref{t1}, it is enough to show that for every $\ep >0$
there is a function $g\colon \nl \to \rr$ such that $\la _2 (\ga ([a,b] ) +
\graph g) <\ep$.

Let $\ep >0$ be given, and let $J_i ,m_i ,c_i$ satisfy conditions (i)-(iii).
Put $F=[a,b]\se \bigcup_{i=1}^\infty J_i$, then $\la _1 (F)=0$ by (i). Since the
coordinate functions of $\ga$  are absolutely continuous, it follows that
$\la _1 (\ga (F)) =0$.

Let $K_1$ be an arbitrary $m_1$-simple set. The total length of the segments
that constitute $K_1$ is $\sqrt{1+m_1^2}$. Then, by (ii), we have
$\la _2 (\ga (\cl J_1 )+K_1 )\le c_1 \cd \sqrt{1+m_1^2}$. We also have $\la _2 (\ga (F)+K_1 )=0$, as $\la _1 (\ga (F))=0$.

Since $\ga (\cl J_1 )+K_1$ and $\ga (F)+K_1$ are both compact, there is an
$\eta _1$ such that  $\la _2 (\ga (\cl J_1 )+K)\le c_1 \cd \sqrt{1+m_1^2} +
\ep /2$ and $\la _2 (\ga (F)+K)<\ep$ whenever $K\in \kk$ and $d(K,K_1 )<
\eta _1$.

Applying Lemma \ref{l3}, we can find an $m_2$-simple set $K_2 \in \kk$ such
that $d(K_2 ,K_1 )<\min (\ep /4, \eta _1 /2)$.
The total length of the segments that constitute $K_2$ is $\sqrt{1+m_2^2}$.
Thus $\la _2 (\ga (\cl J_2 )+K_2 )\le c_2 \cd \sqrt{1+m_2^2}$.
Since $\ga (\cl J_2 )+K_2$ is compact, there is an $\eta _2 \in (0,\eta _1 /2)$
such that $\la _2 (\ga (\cl J_2 )+K)\le c_2 \cd \sqrt{1+m_2^2} +\ep /4$ whenever
$K\in \kk$ and $d(K,K_2 )<\eta _2$.

Continuing the process, we find the sets $K_n$ for every $n=1,2,\ldots$.
It is easy to check that the sets $K_n$ converge to a set $K$ in the
metric space $\kk$, and that
$$\la _2 (\ga ([a,b]) +K)<\sum_{i=1}^\infty \left( c_i \cd \sqrt{1+m_i^2} +
(\ep /2^i )\right) + \la _2 (\ga (F)+K )<3\ep$$
by (iii). Since $K\in \kk$, $K$ contains the graph of a function 
$g\colon \nl \to \rr$, and then $\la _2 (\ga ([a,b] ) + \graph g) <
3\ep$. \hfill $\square$ 

The total variation of a function $f\colon [a,b] \to \rr$ is denoted by
$V(f;[a,b])$.

\begin{lemma} \label{l5}
Let $\ga =(\ga _1 ,\ga _2 )\colon [a,b] \to \sik$ be a continuous curve.
\begin{enumerate}[{\rm (i)}]
\item If $s$ is a horizontal segment, then $\la _2 (\ga ([a,b])+s)\le
V(\ga _2 ; [a,b])\cd |s|$.
\item If $s$ is a segment of slope $m$, then
$$\la _2 (\ga ([a,b])+s)\le \frac{|s|}{\sqrt{1+m^2}} \cd V(m\ga _1 -\ga _2 ;
[a,b]).$$
\end{enumerate}
\end{lemma}
\proof
(i) We may assume that the left endpoint of $s$ is the origin. Let $e_c$ denote
the horizontal line $y=c$, and let $k(c)$ denote the number
of points of the curve $\ga$ on $e_c$. It is clear that the one dimensional
measure of $(\ga ([a,b])+s) \cap e_c$ is at most $k(c)\cd |s|$. Therefore, by
Fubini's theorem, we have $\la _2 (\ga ([a,b])+s)\le |s|\cd \int_J k(y)\, dy$,
where $J$ is the range of $\ga _2$. Since $k(y)=|\ga _2 \am (\{ y\} )|$,
we have, by Banach's theorem \cite[Theorem 6.4, Chapter IX, p. 280]{Saks},
$\int_J k(y)\, dy =V(\ga _2 ; [a,b])$.

(ii) Let $m\in \rr$ be given, and put $\al =\arctan m$. Rotating the curve
$\ga$ by $\al$ in the negative direction we obtain the curve
$g =(g_1 ,g_2 )\colon [a,b] \to \sik$. Clearly, $\la _2 (\ga ([a,b])+s)$
equals $\la _2 (g([a,b])+t)$, where $t$ is a horizontal segment of length $|s|$.
Since $g_2 =-\sin \al \cd \ga _1 +\cos \al \cd \ga _2 =-(1/\sqrt{1+m^2}) \cd
(m\ga _1 -\ga _2 )$, the statement follows from (i). \hfill $\square$ 

{\bf Proof of Theorem \ref{t3}.}
Put $\ga '_1 =f$ and $\ga '_2 =g$. Then $f,g\in L_1 [a,b]$, and a.e. point
$x\in (a,b)$ is a Lebesgue point of $f$ and of $g$.
Let $E$ denote the set of points $x\in (a,b)$ such that $x$ is a Lebesgue
point of $f$ and of $g$, satisfying also $f(x)=\ga '_1 (x)\ne 0$.
By assumption, we have $\la _1 ([a,b] \se E)=0$.

We prove that $\ga$ satisfies the conditions of Lemma \ref{l4}. Let $\ep >0$
be given, and let $x_0 \in E$ be arbitrary. Since $x_0$ is a Lebesgue point
of both $f$ and $g$, and $f(x_0 )\ne 0$, it follows that
\begin{equation} \label{e2}
\begin{split}
&\int_J |g(x)-g(x_0 )| \, dx <\frac{\ep}{2} \cd |J| \quad \text{and} \\
&\int_J |f(x)-f(x_0 )| \, dx < \frac{|f(x_0 )|}{1+|g(x_0 )|} \cd 
\frac{\ep}{2} \cd |J|
\end{split}
\end{equation}
for every short enough interval $J$ containing $x_0$. Then, by the Vitali
covering theorem, there are pairwise disjoint open subintervals
$J_1 ,J_2 ,\ldots$ of $[a,b]$ covering a.e. point of $E$, and there are points
$x_i \in J_i$ such that \eqref{e2} holds with $J=J_i$ and $x_0 =x_i$.
Let $m_i =g(x_i )/f(x_i )$ and
$$c_i =\frac{1}{\sqrt{1+m_i^2}} \cd 
\int_{J_i} \left| \frac{g(x_i )}{f(x_i )} f(x)-g(x) \right| \, dx$$
for every $i$. We prove that the intervals $J_i$ and the numbers $m_i ,c_i$
satisfy the conditions of Lemma \ref{l1}. Condition (i) is obviously
satisfied, and (ii) follows from Lemma \ref{l5}. As for (iii), note that
\begin{align*}
\int_{J_i}  \left| \frac{g(x_i )}{f(x_i )} f(x)-g(x) \right| \, dx  \le & 
 \frac{|g(x_i )|}{|f(x_i )|} \cd \int_{J_i} | f(x)-f(x_i )| \, dx +\\
&+ \int_{J_i} |g(x)-g(x_i )| \, dx <\ep \cd |J_i |
\end{align*}
by \eqref{e2}, and thus (iii) follows. \hfill $\square$ 

\section{Proof of Theorem \ref{t5}}\label{s4}
In this section we shall denote by $|A|$ the Lebesgue measure of the set
$A\su \sik$; that is, $|A|=\la _2 (A)$. Note that $|A|$ also denotes the
cardinality of the finite set $A$. It will be clear from the context 
in which sense the notation is used. Our aim is to prove Theorem \ref{t5}.
\begin{lemma} \label{l6}
If $0<\ep <1$ and $n>n_0 (\ep )$, then there is a sequence $\eta =
(\eta _1 \stb \eta _n )\in \{ -1,1\} ^n$ such that for every $k=1\stb n-1$ and
$0\le u<v\le n-k$ the cardinality of each of the sets
\begin{equation}\label{e3}
\begin{split}  
U_1 &=\{ u+1\le i \le v \colon \eta _i =\eta _{i+k} =1\} ,\\
U_2 &=\{ u+1\le i \le v \colon \eta _i =\eta _{i+k} =-1\},\\
U_3 &=\{ u+1\le i \le v \colon \eta _i =1, \ \eta_{i+k} =-1\},\\
U_4 &=\{ u+1\le i \le v \colon \eta _i =-1, \ \eta _{i+k} =1\}
\end{split}
\end{equation}
is less than $\tfrac{v-u}{4} +\ep n$.
\end{lemma}
\proof 
Let $\mu$ be the normalized counting measure on $X= \{ -1,1\} ^n$; that is, let
$\mu (A)=|A|/2^n$ for every $A\su X$. Then $(X,\mu )$ is a probability space
on which the coordinate functions $x_1 \stb x_n$ are independent.
Note that $\int_X \prod_{j=1}^t x_{i_j} \, d\mu =0$ whenever one of the indices
$i_1 \stb i_t$ is distinct from the others (even if the others are not
necessarily distinct).

By $c_1 ,c_2 ,\ldots$ we denote absolute positive constants.
For every $1\le u\le n$ we have
\begin{equation} \label{e4a}
\int_X \left( \sum_{i=1}^{u} x_{i} \right) ^4 \, d\mu =
\sum_{i_1 \stb i_4 =1}^{u} \int_X \prod_{j=1}^4 x_{i_j}  \, d\mu .
\end{equation} 
By the remark above, we have $\int_X \prod_{j=1}^4 x_{i_j} \, d\mu =0$
except when, after a suitable permutation of $i_1 \stb i_4$, we have $i_1 =i_2$
and $i_3 =i_4$. Thus \eqref{e4a} gives
$$\int_X \left( \sum_{i=1}^{u} x_{i} \right) ^4 \, d\mu <c_1 \cd n^2 .$$
Therefore,
$$\mu \left( \left\{ x\in X\colon |x_1 +\ldots +x_u |
\ge (\ep /2) n \right \}\right) <c_2 \cd n^2 /(\ep ^4 n^4 )=c_2 /(\ep ^4 n^2 ).$$
Consequently, the measure of the set of elements $x$ such that
$|x_1 +\ldots +x_u | \ge (\ep /2) n$ holds for at least one $1\le u\le n$ is
less than $c_2 /(\ep ^4 n )$. Then the measure of the set of elements $x$ such
that $|x_{u+1} +\ldots +x_v | \ge \ep n$ holds for at least one pair 
$(u,v)$ is less than $2c_2 /(\ep ^4 n )$.  

We have, for every $1\le u\le n-k$,
\begin{equation} \label{e4}
\int_X \left( \sum_{i=1}^u x_{i} x_{i +k} \right) ^6 \, d\mu =
\sum_{i_1 \stb i_6 =1}^u \int_X \prod_{j=1}^6 x_{i_j} x_{i_j +k} \, d\mu .
\end{equation} 
Again, we have $\int_X \prod_{j=1}^6 x_{i_j} x_{i_j +k} \, d\mu =0$
except when, after a suitable permutation of $i_1 \stb i_6$, we have $i_1 =i_2$,
$i_3 =i_4$ and $i_5 =i_6$. Thus \eqref{e4} gives
$$\int_X \left( \sum_{i=1}^u x_{i} x_{i +k} \right) ^6 \, d\mu <c_3 \cd n^3 .$$
Therefore,
$$\mu \left( \left\{ x\in X\colon \left| \sum_{i=1}^{u} x_{i} x_{i +k} \right|
\ge (\ep /2) n \right \}\right) <c_4 \cd n^3 /(\ep ^6 n^6 )=c_4 /(\ep ^6 n^3 ) .$$
Then the measure of the set of elements $x$ such that
$|\sum_{i=1}^u x_{i} x_{i +k}|\ge (\ep /2) n$ holds for at least one $k$ and 
$u$ is less than $c_4 /(\ep ^6 n)$. Thus the measure of the set of elements $x$
such that $|\sum_{i=u+1}^v x_{i} x_{i +k}|\ge \ep n$ holds for at least one $k$
and a pair $(u,v)$ is less than $2c_4 /(\ep ^6 n)$. 

If $n>c_5 /\ep ^6$, then $2c_2 /(\ep ^4 n ) +2c_4 /(\ep ^6 n)<1$.
Therefore,
there exists a sequence $\eta \in X$ such that for every $1\le k\le n-1$
and $0\le u<v\le n-k$ we have $|\eta_{u+1} +\ldots +\eta_v | < \ep n$,
$|\eta_{u+1+k} +\ldots +\eta_{v+k} | <\ep n$, and
$|\sum_{i=u+1}^v \eta _i \eta _{i+k} |<\ep n$. Then
$$4\cd |U_1 |=\sum_{i=u+1}^v (\eta _i +1)(\eta _{i+k} +1 )< 4\ep n+ (v-u),$$
and $|U_1 | <((v-u)/4) +\ep n$. For $j=2,3,4$ we obtain 
$|U_j | <((v-u)/4) +\ep n$ similarly. \hfill $\square$

If $I$ is a bounded closed interval, $I=[c,d]$, then we denote
$$I^{-1}=[c,(c+d)/2], \qquad I^{1}=[(c+d)/2 ,d].$$
If $R=[a,b]\times [c,d]$, $n$ is a positive integer and $\al \in \{ -1,1\} ^n$,
$\al = (\al _1 \stb \al _n )$, then we denote $R^\al =\bigcup_{i=1}^n ([t_{i-1},
t_i ]\times [c,d]^{\al _i} )$, where $t_i =a+\tfrac{i}{n} \cd (b-a)$
$(i=0\stb n)$.
\begin{lemma} \label{l7}
Let $\ep >0$, $\al , \be \in \{ -1,1\} ^n$, $\al = (\al _1 \stb \al _n )$,
$\be = (\be _1 \stb \be _n )$ be given such that for every $k=0\stb n-1$, the
cardinality of each of the sets
\begin{equation}\label{e4b}
\begin{split}  
V_1 &=\{ 1\le i\le n-k \colon \al _{i+k} =\be _i =1\} ,\\
V_2 &=\{ 1\le i\le n-k \colon \al _{i+k} =\be _i =-1\} ,\\
V_3 &=\{ 1\le i\le n-k \colon \al _{i+k} =1, \ \be _i =-1\} ,\\
V_4 &=\{ 1\le i\le n-k \colon \al _{i+k} =-1, \ \be_{i+k} =1\}
\end{split}
\end{equation}
is less than $\tfrac{n-k}{4} +\ep n$. If $R$ is a rectangle with sides parallel
to the axes, then 
\begin{equation} \label{e5}
  |R^\al \cap (R^\be +(x,y)) |< \tfrac{1}{4}\cd |R\cap (R+(x,y))|+\ep \cd |R|
\end{equation}
for every $x\ge 0$ and $y\in \rr$.
g\end{lemma}
\proof  Applying a suitable affine transformation we may assume that $R=\nl ^2$.
Let $F(x,y)=|R^\al \cap (R^\be +(x,y))|$. In order to prove \eqref{e5} we may
assume $y\ge 0$. Indeed, the case of $y<0$ can be reduced to the case
of $y>0$ if we replace $y$ by $-y$, $\al$ by $-\al$ and $\be$ by $-\be$. We
may also assume that $x,y<1$, since otherwise $F(x,y)=0$. Summing up: we
assume $0\le x<1$ and $0\le y<1$.

There is an integer $1\le k\le n$ such that $(k-1)/n\le x \le k/n$.
It is easy to check that for every fixed $y$ the function $x\mapsto F(x,y)$
is linear in the interval $[(k-1)/n,k/n]$. The same is true for the right hand
side of \eqref{e5}. Therefore, the difference of the two sides of \eqref{e5}
is linear in $[(k-1)/n,k/n]$, and thus it is enough to prove \eqref{e5} in
the cases when $x=k/n$ $(k=0\stb n-1)$.

Let $0\le k\le n-1$ be given. Then $0\le y <1$ implies
$$R\cap \left( R + \left( \tfrac{k}{n},y \right) \right) = \left[
  \tfrac{k}{n},1  \right] \times [y,1]$$
and
\begin{equation}\label{e5a}
\left| R\cap \left( R + \left( \tfrac{k}{n},y \right) \right) \right|
=\tfrac{n-k}{n} \cd (1-y).
\end{equation}
We have
\begin{equation}\label{e5b}
R^\al \cap \left( R^\be + \left( \tfrac{k}{n},y \right) \right) =
\bigcup_{i=1}^{n-k} \left( \left[ \tfrac{i+k-1}{n},\tfrac{i+k}{n}\right] 
\times J_i \right) ,
\end{equation}
where
\begin{equation*}
J_i = \nl ^{\al _{i+k}} \cap \left( \nl ^{\be _{i}} +y\right) .
\end{equation*}
By \eqref{e5a} and \eqref{e5b} it is enough to show that
\begin{equation*} 
\sum_{i=1}^{n-k} |J_i |< \frac{n-k}{4} \cd (1-y) + \ep n. 
\end{equation*}   
Now we have
\begin{align*}
  \al _{i+k} =\be _i =1 &\akkor J_i =[1/2,1] \cap [(1/2)+y,1+y] ,\\
  \al _{i+k} =\be _i =-1 &\akkor J_i =[0,1/2] \cap [y,(1/2)+y]  ,\\
  \al _{i+k} =-1, \ \be _i =1 &\akkor J_i =[0,1/2] \cap [(1/2)+y,1+y] ,\\
  \al _{i+k} =1, \ \be _i =-1 &\akkor J_i =[1/2,1] \cap [y,y+(1/2)] .
\end{align*}
Let $V_1 \stb V_4$ be as in \eqref{e4b}, and put $|V_j |=v_j$ $(j=1\stb 4)$.
If $1/2\le y\le 1$, then we obtain
$$  \sum_{i=1}^{n-k} |J_i | \le  v_4 \cd (1-y) < \left( \tfrac{n-k}{4} +\ep n
\right) \cd (1-y)< \tfrac{n-k}{4} \cd (1-y) + \ep  n.$$
If $0\le y\le 1/2$, then
$$\sum_{i=1}^{n-k} |J_i | \le  v_1 \cd ((1/2)-y)+ v_2 \cd ((1/2)-y)+
v_4 \cd y < \tfrac{n-k}{4} \cd (1-y) + \ep n,$$
and the proof is complete. \hfill $\square$

{\bf Proof of Theorem \ref{t5}.} First we construct a Borel measurable
function $g\colon \nl \to \nl$ such that $\graph g$  does not have
property ${\rm (K^t)}$. The graph of the function $g$ will be given -- apart
from a countable set -- as the intersection of a nested sequence of sets
$A_0 \sp A_1 \sp \ldots$ that are finite unions of rectangles. 

We put $n_0 =N_0 =1$ and $A_0 =R^0_1 =\nl ^2$. Suppose that $k\ge 1$, and
the integers $n_0 \stb n_{k-1}$, $N_{k-1} =n_0 \cdots n_{k-1}$, the rectangles
$R^{k-1}_j =\left[ \tfrac{j-1}{N_{k-1}},\tfrac{j}{N_{k-1}} \right] \times J^{k-1}_j$
$(j=1\stb N_{k-1})$ and the set $A_{k-1}= \bigcup_{j=1}^{N_{k-1}} R^{k-1}_j$ have been
defined, where $J^{k-1}_1 \stb J^{k-1}_{N_{k-1}}$ are closed subintervals of $\nl$
of length $1/2^{k-1}$.

Put $\ep _k = 2^{-2k-3}/N_{k-1}$. By Lemma \ref{l6}, we can choose an integer
$N_k > N_{k-1} /\ep _k$ and a sequence $\eta ^k \in \{ -1 ,1\}^{N_k}$ such
that $N_{k-1} \mid N_k$, and $\eta ^k =(\eta ^k_1 \stb \eta ^k_{N_k} )$
has the following property:

For every $k=1\stb N_k -1$ and
$0\le u<v\le n-k$ the cardinality of each of the sets
\begin{equation}\label{e6}
\begin{split}  
 &\{ u+1\le i \le v \colon \eta^k _i =\eta^k _{i+k} =1\} ,\\
 &\{ u+1\le i \le v \colon \eta^k _i =\eta^k _{i+k} =-1\},\\
 &\{ u+1\le i \le v \colon \eta^k _i =1, \ \eta^k_{i+k} =-1\},\\
 &\{ u+1\le i \le v \colon \eta^k_i =-1, \ \eta^k _{i+k} =1\}
\end{split}
\end{equation}
is less than $((v-u)/4) +\ep_k N_k$.

Putting $n_k =N_k /N_{k-1}$ we get $N_k =n_0 \cd n_1 \cdots n_k$ and $n_k >
1/\ep _k$. We define 
$$\al ^k_j = (\eta ^k_{(j-1)\cd n_k +1} \stb \eta ^k_{j\cd n_k}) \qquad
(j=1\stb N_{k-1} )$$
and
$$A_k= \bigcup_{j=1}^{N_{k-1}} \left( R^{k-1} _j \right) ^{\al ^k_j} .$$
Clearly, $A_k$ is the union of the rectangles $R^k_j =\left[
\tfrac{j-1}{N_k},\tfrac{j}{N_k} \right] \times J^k_j$
$(j=1\stb N_k )$, where $J^k_1 \stb J^k_{N_k}$ are closed subintervals
of $\nl$ of length $1/2^k$.

In this way we defined the sequences of integers $n_k$, $N_k$ and the 
sequence of sets $A_k$ such that $n_0 =N_0 =1$, $n_k \ge 4$ and $N_k =
n_1\cdots n_k$ if $k=1,2,\ldots$, and $\nl ^2 =A_0 \sp A_1 \sp \ldots$. It
is clear that $|A_k |=1/2^k$ for every $k$.

Let $A=\bigcap_{k=0}^\infty A_k$. Then $A$ is compact. It is easy to check
that $\{ y\colon (x,y)\in A\}$ is a singleton for every irrational $x$,
and has at most two elements if $x$ is rational. Thus there is a Borel
measurable function $g\colon \nl \to \nl$ such that $A \se \graph g$ is
countable. We prove that $A$ does not have property ${\rm (K^t)}$. 
\begin{lemma} \label{l8}
We have 
\begin{equation*} 
\left| A_k \cap \left( A_k + (x,y) \right) \right| < 2^{j-2k}
\end{equation*}   
for every $k\ge 0$, $j=0\stb k$, $x\ge 1/N_j$ and $y\in \rr$.
\end{lemma}
\proof We prove the slightly stronger inequality
\begin{equation} \label{e8}
  \left| A_k \cap \left( A_k + (x,y) \right) \right| \le 2^{j-2k} -2^{-3k}
  \tag{10$_k$}
\end{equation}   
under the same conditions. We prove by induction on $k$.
If $k=0$, then $x\ge 1/N_0 =1$ implies that both sides of (10$_0$) is zero,
and thus (10$_0$) is true. Let $k\ge 1$, and suppose that
(10)$_{k-1}$ holds for every $j,x,y$ satisfying
the conditions stated in the lemma. Since $A_{k-1}=\bigcup_{j=1}^{N_{k-1}} R_j^{k-1}$,
we have
$$A_{k-1} \cap \left( A_{k-1} + (x,y) \right) =\bigcup_{j_1 , j_2 =1}^{N_{k-1}} \left( R_{j_1}^{k-1}
\cap (R_{j_2}^{k-1} +(x,y)) \right)$$
and
\begin{equation} \label{e8b}
\begin{split}
\big| (R_{j_1}^{k-1} \cap A_k ) \cap \big( (R_{j_2}^{k-1} \cap A_k ) &+
(x,y) \big) \big| =\\
&\left| \left( R_{j_1}^{k-1} \right) ^{\al ^k_{j_1}} \cap \left( \left(
R_{j_2}^{k-1}   \right) ^{\al ^k_{j_2}}  + (x,y) \right) \right|
\end{split}
 \tag{11}
\end{equation}
for every $j_1$ and $j_2$. Suppose $x\ge 1/N_k$. We prove
\begin{equation} \label{e8c}
\begin{split}
\bigg| \left( R_{j_1}^{k-1} \right) ^{\al ^k_{j_1}} \cap & \left( \left(
R_{j_2}^{k-1}   \right) ^{\al^k_{j_2}} + (x,y) \right) \bigg| <\\
&\tfrac{1}{4} \cd \left| R_{j_1}^{k-1} \cap (R_{j_2}^{k-1} +
(x,y) )\right| +2^{-2k-2} \cd |R^{k-1}_{j_1} |
\end{split}
\tag{12}
\end{equation}
for every $1\le j_1 , j_2 \le N_{k-1}$. If $j_2 >j_1$, then \eqref{e8c} is true,
because its left hand side is zero.

Suppose $j_2 < j_1$. Since $R_{j_1}^{k-1}$ is a translated copy of $R_{j_2}^{k-1}$,
$(R_{j_1}^{k-1} ) ^{\al^k_{j_1}}$ is a translated copy of $(R_{j_2}^{k-1} ) ^{\al^k_{j_2}}$,
and we may apply Lemma \ref{l7} with $R=R_{j_2}^{k-1}$, $n=n_k$, $\al =
\al^k_{j_2}$ and $\be = \al^k_{j_1}$. Since the cardinality of each of
the sets listed in \eqref{e6} is less than $((v-u)/4) +\ep_k N_k$, and
$\ep_k \cd N_k =(\ep _k\cd N_{k-1} ) \cd n_k$, it follows that the
conditions of Lemma \ref{l7} are satisfied with $\ep =\ep _k \cd N_{k-1}$.
By the choice of $\ep _k$, \eqref{e8c} is true.

If $j_1 =j_2$, then we can argue as follows. Suppose $j_1 =j_2 =j$, and
$\al =\al ^k_j =(\al _1 \stb \al _{n_k})$. Since $x\ge 1/N_k$, we have
$$ \left( R_j^{k-1}   \right) ^{\al} \cap \left( \left( R_j^{k-1}   \right) ^{\al}
+ (x,y) \right) \su \left( R_j^{k-1}   \right) ^{\al} \cap \left( \left( R_j^{k-1}
\right) ^\be  + (x',y) \right),$$
where $x'=x-(1/N_k)\ge 0$ and $\be =(1,\al _1 \stb \al_{n_k-1})$.
Put $\ep =\ep _k \cd N_{k-1}$. If $0\le m \le n_k$, then
\begin{align*}
| \{ 1\le i\le n-m \colon & \al _{i+m} =\be _i =1\}|  \le \\
&1 +((n-m)/4) +\ep n_k < ((n-m)/4) +2\ep n_k ,
\end{align*}
as $n_k>1/\ep _k\ge 1/\ep$. We get the same estimate for the cardinality
of the sets $V_2 , V_3 , V_4$ of Lemma \ref{l7}. Applying the lemma, we obtain
\eqref{e8c}  in this case as well.

Now we add \eqref{e8c} for every $j_1$ and $j_2$. Since the (possibly
degenerate) rectangles $R_{j_1}^{k-1} \cap (R_{j_2}^{k-1} +(x,y) )$
are pairwise nonoverlapping and $\sum_{j=1}^{N_{k-1}} |R^{k-1}_j |=|A_{k-1} |$,
we obtain
\begin{equation} \label{e9}
\left| A_k \cap \left( A_k + (x,y) \right)   \right| <
\tfrac{1}{4} \cd \left| A_{k-1} \cap \left( A_{k-1} + (x,y) \right)   \right|
+2^{-2k-1} \cd |A_{k-1} |. \tag{13}
\end{equation}
(Note that for every $j_1$ there are at most two $j_2$ such that
$|R_{j_1}^{k-1} \cap (R_{j_2}^{k-1} +(x,y) )|>0$.) By $|A_{k-1} |=2^{-k+1}$, the
right hand side of \eqref{e9} is at most $2^{-k-1} +2^{-2k-1} \cd 2^{-k+1} \le
2^{-k} -2^{-3k}$. Thus (10)$_k$ is proved if $x\ge 1/N_k$.

If $x\ge 1/N_j$, where $j<k$, then (10)$_{k-1}$ holds by
the induction hypothesis, and thus the left hand side of \eqref{e9} is less than
$$\frac{1}{4} \cd \left( 2^{j-2k+2} -2^{-3k+3} \right) +
2^{-2k-1} \cd 2^{-k+1} \le 2^{j-2k} -2^{-3k} .$$
 \hfill $\square$

\bigskip
Now we prove that $A$ does not have property ${\rm (K^t)}$. Let $h\colon \nl
\to \rr$ be a simple function, and let $B=A+\graph h$.
We prove that $|B|\ge 1/9$. Let $H$ be the closure of $\graph h$, then
$H\se \graph h$ is finite. Since $|A|=0$, it is enough to show that
$|A+H|\ge 1/9$. Since $A$ and $H$ are compact, we have
$$A+H=\bigcap_{k=0}^\infty (A_k +H),$$
and thus it is enough to prove that $|A_k +H|\ge 1/9$ for every $k$.

Let $k$ be fixed, and put $C_i =A_k +\{ (i/N_k , h(i/N_k ))\}$ and $f_i =
\chi _{C_i}$ for every $i=1\stb N_k$. Then $C_i \su A_k +H$ for every $i$, and
$$\int_{\sik} f_{i_1} f_{i_2} \, d\la _2 = |C_{i_1} \cap C_{i_2} | =
\left| A_k \cap (A_k +((i_2 -i_1 )/N_k ,y))\right| ,$$
where $y=h(i_2 /N_k )-h(i_1 /N_k )$. By Lemma \ref{l8}, we obtain
$\int_{\sik} f_{i_1} f_{i_2} \, d\la _2 < 2^{j-2k}$ whenever $i_2 -i_1 \ge N_k /N_j$.
Let $I_j$ denote the set of pairs $(j_1 ,j_2 )$ such that $1\le i_1 <i_2
\le N_k$ and $N_k /N_j \le i_2 -i_1 < N_k /N_{j-1}$. Then we have, for every
$j=1\stb k$,
\begin{align*}
\int_{\sik} \sum_{(i_1 ,i_2 )\in I_j} f_{i_1} f_{i_2} \, d\la _2 &<|I_j |\cd 2^{j-2k}
<N_k \cd \frac{N_k}{N_{j-1}} \cd 2^{j-2k}\\
& \le \frac{N_k^2}{4^{j-1}}\cd 2^{j-2k}=4N_k^2 \cd 2^{-j-2k}
\end{align*}
by $N_j \ge 4^j$. Therefore, we obtain
\begin{equation*}\label{10}
\begin{split}
\int_{\sik} \left( \sum_{i=1}^{N_k} f_i \right) ^2 \, d\la _2 =&
\int_{\sik} \sum_{i=1}^{N_k} f^2_i \, d\la _2  +2\cd \int_{\sik}
\sum_{1\le i_1< i_2 \le N_k} f_{i_1} f_{i_2}\, d\la _2  <\\
&N_k \cd 2^{-k} + 2\cd \sum_{j=1}^k 4N_k^2 \cd 2^{-j-2k} <9N_k^2 \cd 2^{-2k}.
\end{split}
\end{equation*}
Put $C=\bigcup_{i=1}^{N_k} C_i$ and $f=\chi _C$. Then $C\su A_k +H$. We have
\begin{equation*}\label{e11}
\begin{split}
N_k \cd 2^{-k} =&\int_{\sik} \sum_{i=1}^{N_k} f_i \, d\la _2 =
\int_{\sik} \left( \sum_{i=1}^{N_k} f_i \right) \cd f\, d\la _2 \le \\
&\left( \int_{\sik} \left( \sum_{i=1}^{N_k} f_i \right) ^2 \,
d\la _2 \right)^{1/2} \cd \left( \int_{\sik} f^2 \, d\la _2 \right)^{1/2} <\\
& 3N_k \cd 2^{-k} \cd \sqrt{|C|} ,
\end{split}
\end{equation*}
and thus $|C| >1/9$. Then $|A_k +H|\ge 1/9$ for every $k$,
and $|A +H|\ge 1/9$. Since $A\se (\graph g)$ is countable,
$H\se (\graph h)$ is finite and $|H|=0$, we obtain $|(\graph g) +
(\graph h )|\ge 1/9$. This proves that $\graph g$ does not have property
${\rm (K^t)}$.

Finally, we show that a slight modification of the construction of the sets
$A_k$ gives a continuous function $f$ not having property ${\rm (K^t)}$.
We fix a sequence of numbers $q_k$ such that $q_k >1$ for every $k=0,1,\ldots$,
and $\prod_{k=0}^\infty q_k <2$.

We have $A_k= \bigcup_{j=1}^{N_k} R_j^k$, where the rectangles $R_j^k$
$(j=1\stb N_k )$ are nonoverlapping. Let ${\tilde R}_j^k$ be a rectangle
homothetic to $R_j^k$, covered by the interior of $R_j^k$, and such that 
$|{\tilde R}_j^k |> |R_j^k | /q_k$. We put
${\tilde A}_k = \bigcup_{j=1}^{N_k} {\tilde R}_j^k$ for every $k$, and  
$\tilde A =\bigcap_k^\infty {\tilde A}_k$. Then we have $|{\tilde A}_k |
>2^{-k-1}$ for every $k$.

The section $\{ y\colon (x,y)\in \tilde A \}$ is either empty or a singleton for
every $x$, and thus $\tilde A=\graph f_0$, where $f_0$ is a continuous function
defined on the projection $D$ of $\tilde A$ onto $\nl$.

Let $h$ and $H$ be as above, let $k$ be fixed, and put ${\tilde C}_i =
{\tilde A}_k + \{ (i/N_k , h(i/N_k ))\}$, $\tilde C =\bigcup_{i=1}^{N_k}
{\tilde C}_i$ and ${\tilde f}_i = \chi _{{\tilde C}_i}$ for every $i=1\stb N_k$.
Since $|{\tilde A}_k |=2^{-k-1}$ for every $k$, the computation
of \eqref{e11} gives
$$N_k \cd 2^{-k-1}<3N_k \cd 2^{-k} \cd \sqrt{|{\tilde C}|} ,$$ 
and thus $|{\tilde C}|>1/36$. This proves that $\graph f_0$ does not have
property ${\rm (K^t)}$. Since $D$ is compact, $f_0$ can be extended to $\nl$ as
a continuous function $f$. It is clear that $\graph f$ does not have
property ${\rm (K^t)}$ either. \hfill $\square$

\subsection*{Acknowledgments}


The first author was supported by the Hungarian National Foundation for
Scientific Research, Grant No. K146922.

\begin{small}\noindent (M. Laczkovich) \\
{\sc Department of Analysis, E\" otv\" os Lor\' and University,\\  
Budapest, P\' azm\' any 
P\' eter s\' et\' any 1/C, 1117 Hungary. \\
email: {\tt miklos.laczkovich@gmail.com} 
 
\bigskip\noindent (A. M\'ath\'e)\\
Mathematics Institute, University of Warwick\\
email: {\tt a.mathe@warwick.ac.uk}
}
\end{small}

\end{document}